\documentclass{article}

\usepackage[english]{babel}
\usepackage[utf8x]{inputenc}
\usepackage[T1]{fontenc}

\usepackage[a4paper,includefoot,top=3cm,bottom=2cm,left=3cm,right=3cm,marginparwidth=1.75cm]{geometry}

\usepackage{amsmath}
\usepackage{graphicx}
\usepackage[colorlinks=true, allcolors=blue]{hyperref}
\usepackage{enumerate}
\usepackage{amsmath,amssymb,amsthm}
\usepackage{float}

\numberwithin{equation}{section}

\newcommand{\stkout}[1]{\ifmmode\text{\sout{\ensuremath{#1}}}\else\sout{#1}\fi}   

\newtheorem{theorem}{Theorem}[section]

\newtheorem{lemma}[theorem]{Lemma}
\theoremstyle{definition}

\newtheorem{assumption}[theorem]{Assumption}
\theoremstyle{remark}

\title{Large deviations related to the law of the iterated logarithm for It\^{o} diffusions\thanks{We gratefully acknowledge financial support from the Austrian
Science Fund (FWF) under grant P30750.}}
\author{Stefan Gerhold \\ TU Wien \\ sgerhold@fam.tuwien.ac.at
\and Christoph Gerstenecker \\ TU Wien \\ christoph.gerstenecker@fam.tuwien.ac.at}

\begin{document}
	
\maketitle

\begin{abstract}
	When a Brownian motion is scaled according to the law of the iterated logarithm,
	its supremum converges to one as time tends to zero. Upper large deviations of the supremum
	process can be quantified by writing the problem in terms of hitting times and applying
	a result of Strassen (1967) on hitting time densities.
	We extend this to a small-time large deviations principle for the supremum
	of scaled It\^{o} diffusions,
	using as our main tool a refinement of Strassen's result due to Lerche (1986).
\end{abstract}

\section{Introduction and main results}\label{sec:introduction and main results}

For a standard Brownian motion~$W$ and
\begin{align*}
		h(u)
		:= \sqrt{2 u \log\log \frac{1}{u}},
\end{align*}
we have
\[
\limsup_{t\searrow0}\frac{W_t}{h(t)}=
  \lim_{t\searrow0} \sup_{0<u<t} \frac{W_u}{h(u)} = 1 \quad \text{a.s.,}
\]
by Khinchin's law of the iterated logarithm, and there are extensions to the diffusion case. In this note we are not interested in a.s.\ convergence,
but rather in small-time large deviations of the process $\sup_{0<u<t} X_u/h(u)$
for an It\^{o} diffusion~$X$. For Brownian motion, a large deviations
estimate follows from a result of Strassen~\cite{St67},
which gives precise tail asymptotics for the last (or, by time inversion, first) time
at which a Brownian motion hits a smooth curve. For fixed $\varepsilon>0$, it yields
\begin{equation}
  P\left(\sup_{0<u<t} \frac{W_u}{h(u)} \geq \sqrt{1+\varepsilon}\right)
  = e^{-\varepsilon (\log \log \frac1t) (1+o(1))}, 
  \quad t\searrow 0. \label{eq:est intro}
\end{equation}
See Section~\ref{se:bm} for details.
In Theorem~\ref{thm:lerche theorem 4.1.} below, we cite an extension
of Strassen's result due to Lerche~\cite{Le86}, which we will use when extending
the estimate~\eqref{eq:est intro} to It\^{o} diffusions. 
We make the following mild assumptions on our diffusion process.
\begin{assumption}\label{thm:assumptions}
	\begin{enumerate}[(i)]
	   \item\label{it:sde}
	     The continuous one-dimensional stochastic process $X=(X_t)_{t\geq0}$ satisfies the SDE
      \begin{align*}
	      X_t =& \int_{0}^{t} b(X_u, u)\, du + \int_{0}^{t} \sigma(X_u, u)\, dW_u,
      \end{align*}
	    \item\label{it:cont} the coefficients $b$ and  $\sigma$  are continuous functions from
	    $\mathbb{R}\times [0,\infty)$ to $\mathbb R$ with
	    \[
	      \sigma_0 := \sigma(X_0,0) = \sigma(0,0) >0,
	    \]
		\item \label{it:ld assumption for s} the supremum of $ X $ satisfies
		a weak form of
		a small-time large deviations estimate, in the sense that there are $c_1,c_2>0$ such that
			\begin{align}\label{eq:large deviation assumtion}
				P\left(\sup_{0 < u < t} |X_u| \geq c_1\right)
				=O(t^{c_2}), \quad t \searrow 0,
			\end{align}
	   \item \label{it:lil}
	   the process~$X$ satisfies the small-time law of the iterated logarithm, i.e.,
	   \begin{align*}
	       \limsup_{t\searrow0}\frac{X_t}{h(t)}=   
	      \lim_{ t \searrow 0} \sup_{0<u<t}\frac{X_u}{h(u)} = \sigma_0,\quad \text{a.s.}
	   \end{align*}
	\end{enumerate}
\end{assumption}
By inspecting our proofs (see Lemma~\ref{le:drift}
and~\eqref{eq:use g}), it is not hard to see that
the continuity assumption~\eqref{it:cont} can be slightly weakened.
As for part~\eqref{it:ld assumption for s}, note that it is much weaker than
a classical large deviations estimate, with exponential decay rate. The latter holds, e.g., under the conditions of the
small-noise LDP in~\cite{BaCa11}, by applying Brownian scaling and the contraction principle. For sufficient conditions for the law of the iterated logarithm,
we refer to  p.57 in~\cite{Mc69} and p.11 in~\cite{Ca98}.

\begin{theorem}\label{thm:main result}
  Under  
  Assumption~\ref{thm:assumptions},
  the process $\, \sup_{0 < u < t} X_u/h(u)$ satisfies
  a small-time large deviations principle with speed $\log \log (1/t)$ and rate function
  \[
    J(x) := 
    \begin{cases}
          (x/\sigma_0)^2-1 & x  \geq \sigma_0, \\
          \infty & x < \sigma_0.
    \end{cases}
  \]
	This means that
	\begin{align}\label{eq:ldp lower}
		 \liminf_{t \searrow 0} \frac{1}{\log\log\frac{1}{t}} \log P\left( \sup_{0 < u < t} \frac{X_u}{h(u)} \in O \right) &\geq -J(O)
	\end{align}
	for any open set $ O $ and
	\begin{align}\label{eq:ldp upper}
		\limsup_{t \searrow 0} \frac{1}{\log\log\frac{1}{t}} \log P\left( \sup_{0 < u < t} \frac{X_u}{h(u)} \in C \right)
		&\leq -J(C)
	\end{align}
	for any closed set $ C $, where $J(M) := \inf_{x \in M} J(x).$
\end{theorem}

Obviously, $J$ is a \emph{good rate function} in the sense of~\cite{DeZe98},
i.e.\ the level sets $\{J\leq c\},$ $c\in\mathbb R,$ are compact.
The main estimate needed to prove Theorem~\ref{thm:main result} is contained in the following result.

\begin{theorem}\label{thm:main eps}
	Under parts \eqref{it:sde}--\eqref{it:ld assumption for s} of
	Assumption~\ref{thm:assumptions}, for $\varepsilon>0$ we have
	\begin{align*}
		P\left( \sup_{0< u < t}\frac{X_u}{h(u)} \geq \sigma_0\sqrt{1 + \varepsilon} \right)
		&= e^{- \varepsilon (\log\log\frac{1}{t})(1 + o(1))} 
		\\
		&= \Big(\log \frac1t\Big)^{- \varepsilon + o(1)}, \quad t \searrow 0.
	\end{align*}
\end{theorem}

After some preparations, the proofs of Theorems~\ref{thm:main result} and~\ref{thm:main eps}
are given at the end of Section~\ref{se:main proofs}.

\section{Brownian motion}\label{se:bm}

We can quickly see that there are positive constants $\gamma_1,\gamma_2$ (depending
on~$\varepsilon$) such that
\begin{equation}\label{eq:c1 c2}
   e^{-\gamma_1 (\log \log \frac1t)(1+o(1))} \leq
   P\left(\sup_{0<u<t} \frac{W_u}{h(u)} \geq \sqrt{1+\varepsilon}\right)
   \leq e^{-\gamma_2 (\log \log \frac1t)(1+o(1))}, \quad t\searrow0.
\end{equation}
As for the lower estimate, note that~$h(u)$ increases for small $u>0$, and thus
\[
  P\left(\sup_{0<u<t} \frac{|W_u|}{h(u)} \geq \sqrt{1+\varepsilon}\right)
  \geq
  P\left(\sup_{0<u< t}|W_u|\geq \sqrt{1+\varepsilon}\, h(t)\right),
  \quad t\text{ small.}
\]
{}From this and the reflection principle, it is very easy to see that we can take $\gamma_1=\varepsilon+1$
in~\eqref{eq:c1 c2}.
The upper estimate in~\eqref{eq:c1 c2} follows from applying the Borell inequality (Theorem~D.1 in~\cite{Pi96})
to the centered Gaussian process $(W_u/h(u))_{0<u<t},$ but neither
of these estimates is sharp.
To get the optimal constants $\gamma_1=\gamma_2=\varepsilon$, we use a result of Strassen~\cite{St67}.
By time inversion, we have
\begin{align}
  P\left(\sup_{0<u<t} \frac{W_u}{h(u)} \geq \sqrt{1+\varepsilon}\right)
  &= P\big( \inf\{u: W_u \geq \sqrt{1+\varepsilon}\, h(u) \} \leq t\big) \notag \\
  &= P\left( \sup\{v: W_v \geq \sqrt{1+\varepsilon}\, vh(1/v) \} \geq\frac{1}{t}\right). \notag
\end{align}
Define $\varphi(v)=\sqrt{1+\varepsilon}\, vh(1/v)$. Then, by Theorem~1.2 of~\cite{St67},
the random variable $\sup\{v: W_v \geq \varphi(v) \}$ has a density $D_\varphi(s)$ (except possibly for some
mass at zero, which is irrelevant for our asymptotic estimates), which satisfies
\[
   D_\varphi(s) \sim \varphi'(s) (2\pi s)^{-1/2}\exp\big({-\varphi(s)^2/2s}\big),
   \quad s\nearrow \infty.
\]
{}From this, the estimate~\eqref{eq:est intro} easily follows, very similarly as in the proof
of Theorem~\ref{thm:bm with o(1)} below. That theorem strengthens~\eqref{eq:est intro}, replacing $\varepsilon$  by some quantity that converges to $\varepsilon$. To prove it,
we apply the following theorem due to Lerche:
\begin{theorem}[Theorem 4.1 in~\cite{Le86}, p.60]\label{thm:lerche theorem 4.1.}
	Let $ T_a := \inf\{ u > 0 : W_u \geq \psi_a(u) \} $ for some positive, increasing, continuously
	differentiable function $u\mapsto \psi_a(u),$ which depends on a positive
	parameter~$a$. Assume that
	there are $ 0 < t_1 \leq \infty $ and $ 0 < \alpha < 1 $ such that
	\begin{enumerate}[(i)]
		\item\label{it:lerche -- convergence} $ P(T_a < t_1) \to 0 $ as $ a \nearrow \infty $,
		
		\item\label{it:lerche -- monotonicity} $ \psi_a(u) / u^{\alpha} $ is monotone decreasing in $ u $ for each $ a $,
		
		\item\label{it:lerche -- continuity} for every $ \varepsilon > 0 $ there exists a $ \delta > 0 $ such that for all $ a $
		\begin{align*}
			\left|\frac{\psi_a'(s)}{\psi_a'(u)}   - 1\right| < \varepsilon \quad\text{if}\quad \left|\frac{s}{u} - 1\right| < \delta,
		\end{align*}
		for $ s,u \in (0, t_1) $.
	\end{enumerate}
	Then the density of~$T_a$ satisfies
	\begin{equation}\label{eq:lerche dens}
		p_a(u)
		= \frac{\Lambda_a(u)}{u^{3/2}} n\Big(\frac{\psi_a(u)}{\sqrt{u}}\Big)\left(1 + o(1)\right)
	\end{equation}
	uniformly on $ (0, t_1) $ as $ a \nearrow \infty $. Here, $ n $ is the Gaussian density
	\begin{align*}
		n(x)
		= \frac{1}{\sqrt{2\pi}}e^{-x^2/2},
	\end{align*}
	and $ \Lambda_a $ is defined by
	\begin{align*}
		\Lambda_a(u)
		:= \psi_a(u) - u \psi_a'(u).
	\end{align*}
\end{theorem}

We can now prove the following variant of Theorem~\ref{thm:main eps},
where~$X$ is specialized to Brownian motion, but $\varepsilon$
is generalized to $\varepsilon+o(1)$.

\begin{theorem}\label{thm:bm with o(1)}
  Let $d(t)$ be a deterministic function with $d(t)=o(1)$ as $t\searrow0$. Then,
  for $\varepsilon>0,$
  \begin{equation}\label{eq:bm with o(1}
     P\left( \sup_{0 < u < t} \frac{ W_{u} }{h(u)} \geq  \sqrt{1 + \varepsilon + d(t)}\right)
     = e^{-\varepsilon (\log \log \frac1t) (1+o(1))},   \quad t\searrow 0.
  \end{equation}
\end{theorem}
\begin{proof}
  We put
  \begin{equation}\label{eq:def q}
    q(t) := \sqrt{1+\varepsilon+d(t)}
  \end{equation}
  and $a=1/t,$ to make the notation similar to~\cite{Le86}.
  We can write the probability in~\eqref{eq:bm with o(1} as a boundary crossing probability,
	\begin{align}
		P\left( \sup_{0 < u < t} \frac{ W_{u}}{h(u)} \geq q(t)\right)
		&= P\left(\inf \left\{ u > 0 : W_u \geq q(1/a)  h(u)  \right\}   < \frac{1}{a}  \right) \notag\\
		&= P\left(\inf \left\{ au > 0 : W_u \geq q(1/a)  h(u)  \right\}   < 1  \right) \notag\\
		&= P\left(\inf \left\{ s > 0 : W_{s/a} \geq q(1/a)  h(s/a)  \right\}   < 1  \right) \notag\\
		&= P\left(\inf \left\{ s > 0 : \sqrt{a}W_{s/a} \geq q(1/a)  \sqrt{a}h(s/a)  \right\}   < 1  \right) \notag\\
		 &= P\left(\inf \left\{ s > 0 : W'_{s} \geq q(1/a)  \sqrt{a}h(s/a)  \right\}   < 1  \right), \label{eq:fuer lerche}
\end{align}
	where $ W' $ is again a Brownian motion, using the scaling property.	
	We will verify in Lemma~\ref{le:lerche} below that the function
    \begin{align}\label{eq:def psi}
		\psi_a(u)
		:= q(1/a) \sqrt{a}h(u/a)
	\end{align}
	satisfies the assumptions of Theorem~\ref{thm:lerche theorem 4.1.}.
  By~\eqref{eq:fuer lerche} and the uniform estimate~\eqref{eq:lerche dens}, we thus obtain
  \[
    P\left( \sup_{0 < u < t} \frac{ W_{u}}{h(u)} \geq q(t)\right)
    \sim
    \int_0^1 \frac{\Lambda_a(u)}{u^{3/2}} n\Big(\frac{\psi_a(u)}{\sqrt{u}}\Big)du,
    \quad a=\frac1t \nearrow \infty.
  \]
    An easy calculation shows that
    \[
      \Lambda_a(u) \sim  \mathit{const} \cdot \sqrt{u \log \log \frac{a}{u}},
    \]
    uniformly in $u\in(0,1),$ and so
    \begin{align*}
      \int_0^1 \frac{\Lambda_a(u)}{u^{3/2}} n\Big(\frac{\psi_a(u)}{\sqrt{u}}\Big)du
      &\sim \mathit{const}\cdot \int_0^1 \frac{1}{u}\sqrt{\log \log \frac{a}{u}} \Big(\log \frac{a}{u}
        \Big)^{-(1+\varepsilon+d(t))} du\\
      &= \mathit{const}\cdot \int_a^\infty \frac{1}{x} \sqrt{\log \log x}\,
        (\log x)^{-(1+\varepsilon+d(t))} dx \\
      &= \mathit{const}\cdot \int_a^\infty \frac{1}{x} 
        (\log x)^{-(1+\varepsilon+o(1))} dx \\
      &= \mathit{const}\cdot (\log a)^{-\varepsilon+o(1)}
      = e^{-\varepsilon(\log \log \frac1t)(1+o(1))}.
    \end{align*}
    As for the third line, note that
    \[
      \log \log x = (\log x)^{\frac{\log \log \log x}{\log \log x}},
    \]
    and that the exponent is $o(1)$ for $x\geq a$ and $a\nearrow\infty$.
\end{proof}

\begin{lemma}\label{le:lerche}
   The function $\psi_a$ defined in~\eqref{eq:def psi}, with~$q$ defined in~\eqref{eq:def q},
   satisfies the assumptions of Theorem~\ref{thm:lerche theorem 4.1.}.
\end{lemma}	
\begin{proof}
    To verify condition \eqref{it:lerche -- monotonicity} of Theorem~\ref{thm:lerche theorem 4.1.},
	it suffices to note that $ h(u) / u^{\alpha} $ decreases for small~$u$ and $\alpha \in (\tfrac12,1)$.
	The continuity condition \eqref{it:lerche -- continuity} easily follows from
	\begin{align*}
		\log(t) \sim \log(T), \quad t / T \nearrow 1, \quad t, T \nearrow \infty.
	\end{align*}
    It remains to show condition~\eqref{it:lerche -- convergence}, i.e., that
	\begin{align}
		P(T_a < 1)
		&= P\left(\inf \left\{ s > 0 : W'_s \geq q(1/a) 
		\sqrt{a} h(s/a) \right\} < 1\right) \notag\\
		&= P\left( \sup_{0 < s \leq 1} \frac{W_s'}{\sqrt{2 s \log\log\frac{s}{a}}} \geq q(1/a) \right)
		 \label{eq:Ta}
	\end{align}
	converges to zero as $a\nearrow \infty$.
	Choose $ a_0 >0$ such that
	\begin{align}\label{eq:p est}
		q(1/a) \geq \sqrt{1 + \frac{2}{3}\varepsilon}, \quad a\geq a_0.
	\end{align}
	By the law of the iterated logarithm for Brownian motion, we have
	\begin{align*}
	  \lim_{s_0\searrow 0} \sup_{0 < s \leq s_0} \frac{|W'_s|}{\sqrt{2 s \log\log \frac{a_0}{s}}} = 1
	  \quad \text{a.s.}
	\end{align*}
	{}From this we get that there exists an $ s_0 >0$ such that
	\begin{align*}
		\sup_{0 < s \leq s_0} \frac{|W'_s|}{\sqrt{2 s \log\log \frac{a_0}{s}}}
		\leq \sqrt{1 + \frac{1}{2} \varepsilon} \quad \text{a.s.}
	\end{align*}
	By monotonicity w.r.t.~$a$, we obtain
	\begin{align}
		\frac{|W'_s|}{\sqrt{2 s \log\log \frac{a}{s}}}
		\leq \frac{|W'_s|}{\sqrt{2 s \log\log \frac{a_0}{s}}}
		\leq \sqrt{1 + \frac{1}{2} \varepsilon}, \quad  a \geq a_0,\ s\in(0,s_0]\quad \text{a.s.}
		 \label{eq:12eps}
	\end{align}
	For $ s \in [s_0, 1] $, note that the first factor of
	\begin{align*}
		\frac{W_s}{\sqrt{2s}} \cdot \frac{1}{\sqrt{\log\log\frac{a}{s}}}
	\end{align*}
	 is bounded pathwise, and that the second factor satisfies
	\begin{align*}
    	\frac{1}{\sqrt{\log\log\frac{a}{s}}}
    	=\frac{1}{\sqrt{\log\log a + o(1)}}
		\to 0, \quad a \nearrow \infty,
	\end{align*}
	uniformly on $[s_0,1]$. From this and~\eqref{eq:12eps}, we get
	\begin{align*}
	  \limsup_{a\nearrow \infty} \sup_{0 < s \leq 1}
	     \frac{W_s'}{\sqrt{2 s \log\log\frac{s}{a}}}
	     \leq \sqrt{1 + \frac{1}{2} \varepsilon},
	\end{align*}
	and together with~\eqref{eq:p est} this implies that~\eqref{eq:Ta} converges to zero.
\end{proof}

\section{It{\^o} diffusions}\label{se:main proofs}

We now show that our results about It{\^o} diffusions can be 
reduced to the case of Brownian motion, which was handled
in the preceding section.
The drift of~$X$ can be easily controlled by continuity and part \eqref{it:ld assumption for s}
of Assumption~\ref{thm:assumptions}. Define
\begin{align}\label{eq:def D}
	D_{t}
	:= \sup_{0 < u < t} \frac{|\int_{0}^{u} b(X_v, v)\, dv|}{h(u)}.
\end{align}

\begin{lemma}\label{le:drift}
   Under parts \eqref{it:sde}--\eqref{it:ld assumption for s}
of	Assumption~\ref{thm:assumptions}, we have
   	\begin{align}\label{eq:lemma 2}
   	   P\left( D_t > \sqrt{t} \right) 
   	   = O(t^{c_2}), \quad t\searrow0.
	\end{align}
\end{lemma}

\begin{proof}
   The continuous function~$b$ is bounded by some constant $c$ on $[-c_1,c_1]\times [0,t],$
   independently of~$t$ for~$t$ small enough. Therefore, if $\sup_{0 < u < t} |X_u| < c_1,$ then
   \[
     D_t \leq \sup_{0<u<t} \frac{cu}{h(u)} = c \sqrt{\frac{t}{2\log \log \frac1t}},
   \]
   which implies
   \[
     P\left( D_t > \sqrt{t}, \sup_{0 < u < t} |X_u| < c_1 \right) = 0,\quad t \text{ small.}
   \]
   The assertion thus follows from~\eqref{eq:large deviation assumtion}.
\end{proof}

Note that the decay rate in~\eqref{eq:lemma 2} is $e^{-c_2 \log(1/t)}$, and thus
negligible in comparison to~\eqref{eq:est intro}.
The next step in the proof of Theorem~\ref{thm:main eps} is contained in 
Lemma~\ref{le:tc bm}, which allows us to deal with the local martingale part, after expressing it as a time-changed
Brownian motion. We will require the following well-known result:

\begin{theorem}[L\'evy modulus of continuity, Theorem 2.9.25 in~\cite{KaSh91}]\label{thm:levy}
  For $f(\delta):= \sqrt{2\delta \log(1/\delta)},$ we have
  \[
    P\left( \limsup_{\delta\searrow0} \frac{1}{f(\delta)}
      \max_{\substack{0\leq s < t \leq 1\\ |t-s|\leq\delta}}
      |W_t-W_s|=1\right) = 1.
  \]
\end{theorem}

\begin{lemma}\label{le:tc bm}
  Suppose that parts \eqref{it:sde}--\eqref{it:ld assumption for s}
of	Assumption~\ref{thm:assumptions} hold.
  Let $\widehat{W}$ be a standard Brownian motion, and $d(t)$ a deterministic function
  satisfying $d(t)=o(1)$ as $t\searrow0$. Then
  \begin{align}
    P\left(\sup_{0 < u < t} \frac{\big|\widehat{W}_{{\langle X \rangle}_u}\big|}{h(u)}  \geq \sigma_0 \sqrt{1 + \varepsilon}+d(t) \right)
    &=e^{- \varepsilon (\log\log\frac{1}{t})(1 + o(1))}, \label{eq:with abs} \\
    P\left(\sup_{0 < u < t} \frac{\widehat{W}_{{\langle X \rangle}_u}}{h(u)} 
     \geq \sigma_0\sqrt{1 + \varepsilon } + d(t)\right)
    &=e^{- \varepsilon(\log\log\frac{1}{t})(1 + o(1))}, \quad t\searrow 0.
    \label{eq:without abs}
  \end{align}
\end{lemma}
\begin{proof}
   By~\eqref{eq:large deviation assumtion}, we may assume $\sup_{0 < u < t} |X_u| < c_1.$
   Define 
	\begin{align*}
		g(u) := \sup_{\substack{|x| \leq c_1\\ s < u}}\left| \sigma^2(x,s) - \sigma_0^2 \right| = o(1), \quad u \searrow 0.
	\end{align*}
	Since
	\begin{align*}
		{\langle X \rangle}_u
		= \int_{0}^{u} \sigma^2(X_v, v)\, dv,
	\end{align*}
	we have
	\begin{align}\label{eq:from mvt}
	   \left| \langle X \rangle_u -  \sigma_0^2 u \right| \leq u g(u)
	\end{align}
	by the mean value theorem.
	For arbitrary $u>0$ and small $\delta>0$, we have
	\begin{align}\label{eq:from levy}
      \max_{\substack{0\leq x < y \leq u\\ |y-x|\leq\delta u}}
      |\widehat{W}_y-\widehat{W}_x| \leq \sqrt{u} \sqrt{3\delta \log(1/\delta)}
    \end{align}
	by Theorem~\ref{thm:levy}
	and Brownian scaling. From~\eqref{eq:from mvt} and~\eqref{eq:from levy},
	we obtain
	\begin{align}
		\sup_{0 < u < t} \frac{\big|\widehat{W}_{{\langle X \rangle}_u} - \widehat{W}_{\sigma_0^2 u}\big|}{h(u)}
		\leq   \sup_{0<u < t} \frac{\sqrt{u}\sqrt{3 g(u) \log\frac{1}{g(u)}}}{h(u)}
		=: r(t)
		= o(1), \quad t \searrow 0, \label{eq:use g}
	\end{align}
	on the event $\sup_{0 < u < t} |X_u| < c_1.$ This implies
	\begin{align}
      P\Bigg(\sup_{0 < u < t} \frac{\big|\widehat{W}_{{\langle X \rangle}_u}\big|}{h(u)}
          \geq & \sigma_0\sqrt{1 + \varepsilon} +d(t),\ \sup_{0 < u < t} |X_u| < c_1\Bigg) \notag\\
      &\leq
      P\left( \sup_{0 < u < t} \frac{\big| \widehat{W}_{\sigma_0^2 u} \big|}{h(u)} \geq \sigma_0 \sqrt{1 + \varepsilon} +d(t) - r(t)\right) \notag\\
      &=
       P\left( \sup_{0 < u < t} \frac{\big| \widetilde{W}_{u} \big|}{h(u)} \geq  \sqrt{1 + \varepsilon} + \frac{d(t) - r(t)}{\sigma_0}\right) \notag  \\
      &\leq 2  P\left( \sup_{0 < u < t} \frac{ \widetilde{W}_{u} }{h(u)} \geq  \sqrt{1 + \varepsilon} + \frac{d(t) - r(t)}{\sigma_0}\right).\notag
    \end{align}
    where $\widetilde{W}$ is again a Brownian motion.
    Now the upper estimate in~\eqref{eq:with abs} follows from Theorem~\ref{thm:bm with o(1)}.
	 To complete the proof of the lemma, a lower estimate for the left hand side of~\eqref{eq:without abs} is needed. We have
	\begin{align*}
	  \sup_{0 < u < t} \frac{\widehat{W}_{{\langle X \rangle}_u}}{h(u)}
	  \geq
	  \sup_{0 < u < t} \frac{\widehat{W}_{\sigma_0^2 u}}{h(u)}
	  -\sup_{0 < u < t} \frac{\big|\widehat{W}_{{\langle X \rangle}_u} - \widehat{W}_{\sigma_0^2 u}\big|}{h(u)},
	\end{align*}
	and thus, by~\eqref{eq:use g},
	\begin{align}
	   P\Bigg(\sup_{0 < u < t} &\frac{\widehat{W}_{{\langle X \rangle}_u}}{h(u)} 
        \geq \sigma_0\sqrt{1 + \varepsilon} +d(t),\ \sup_{0 < u < t} |X_u| < c_1\Bigg) \notag\\
       &\geq
       P\left(\sup_{0 < u < t} \frac{\widehat{W}_{\sigma_0^2 u}}{h(u)} 
        \geq \sigma_0\sqrt{1 + \varepsilon} + d(t)+ r(t),\ \sup_{0 < u < t} |X_u| < c_1 \right) \notag\\
       &\geq
       P\left(\sup_{0 < u < t} \frac{\widetilde{W}_{u}}{h(u)} 
        \geq \sqrt{1 + \varepsilon} + \frac{d(t)+r(t)}{\sigma_0}\right)
          - P\left( \sup_{0 < u < t} |X_u| \geq c_1\right), \label{eq:lower est split}
    \end{align}
    using  $P(A \cap B) \geq P(A) - P(B^c)$. The first probability in~\eqref{eq:lower est split}
    can be estimated by Theorem~\ref{thm:bm with o(1)},
    and the second probability in~\eqref{eq:lower est split}
    is asymptotically smaller by~\eqref{eq:large deviation assumtion}.
\end{proof}

We now conclude the paper by proving our main results,
Theorem~\ref{thm:main eps} and its consequence,
Theorem~\ref{thm:main result}.

\begin{proof}[Proof of Theorem~\ref{thm:main eps}]
   Recalling the definition of~$D_t$ in~\eqref{eq:def D}, we have  
	\begin{align}\label{eq:upper estimate -- first computations}
		P\left(  \sup_{0 < u < t} \frac{X_u}{h(u)} \geq \sigma_0\sqrt{1 + \varepsilon } \right)
		& \leq P\left(  \sup_{0 < u < t} \frac{\left|\int_{0}^{u}\sigma(X_v, v)\, dW_v \right|}{h(u)} + D_t  \geq \sigma_0\sqrt{1 + \varepsilon }   \right).
	\end{align}
	By the Dambis-Dubins-Schwarz theorem (Theorem~3.4.6 and Problem~3.4.7 in~\cite{KaSh91}),
	the local martingale can be written as
	\begin{align}\label{eq:time change}
	  \int_{0}^{u}\sigma(X_v, v)\, dW_v = \widehat{W}_{{\langle X \rangle}_u}
	\end{align}
	with a Brownian motion $\widehat W$.
	The upper estimate thus follows from applying~\eqref{eq:large deviation assumtion}, Lemma~\ref{le:drift},
	and~\eqref{eq:with abs} to~\eqref{eq:upper estimate -- first computations}.	
 We proceed with the lower estimate in Theorem~\ref{thm:main eps}. From
	\begin{align*}
		\sup_{0 < u < t}\frac{X_u}{h(u)}
		\geq \sup_{0< u < t} \frac{\int_{0}^{u}\sigma(X_v,v)\, dW_v}{h(u)}
		- \sup_{0 < u < t}\frac{\left|\int_{0}^{u}b(X_v, v)\, dv\right|}{h(u)}
	\end{align*}
	and~\eqref{eq:time change}, we get
	\begin{align*}
		P\left(  \sup_{0 < u < t} \frac{X_u}{h(u)} \geq \sigma_0\sqrt{1 + \varepsilon}  \right)
		\geq P\left(  \sup_{0 <u < t}\frac{ \widehat{W}_{{\langle X \rangle}_u}}{h(u)} \geq \sigma_0\sqrt{1 + \varepsilon} + D_t \right).
	\end{align*}
	Since we need a lower bound, we can intersect with the event $D_t \leq \sqrt{t}.$
	Using $P(A \cap B) \geq P(A) - P(B^c),$ we obtain
	\begin{align*}
	   P\left(  \sup_{0 < u < t}\frac{ \widehat{W}_{{\langle X \rangle}_u}}{h(u)} \geq \sigma_0\sqrt{1 + \varepsilon} + D_t \right)
	   &\geq
		P\left(  \sup_{0 < u < t}\frac{ \widehat{W}_{{\langle X \rangle}_u}}{h(u)} \geq \sigma_0\sqrt{1 + \varepsilon} + \sqrt{t} ,\ D_t \leq \sqrt{t}\right) \\
	 &\geq	P\left(  \sup_{0 < u < t}\frac{ \widehat{W}_{{\langle X \rangle}_u}}{h(u)} \geq \sigma_0\sqrt{1 + \varepsilon} + \sqrt{t}\right)
		- P(D_t > \sqrt{t}).
	\end{align*}
	The lower estimate now follows from Lemma~\ref{le:drift} and~\eqref{eq:without abs}.
\end{proof}

\begin{proof}[Proof of Theorem~\ref{thm:main result}]
   The increasing process $\sup_{0<u<t}X_u/h(u)$ converges to $\sigma_0$ as $t\searrow0$
   by part~\eqref{it:lil} of Assumption~\ref{thm:assumptions}, and thus its values are $\geq \sigma_0$
   a.s. Hence, there are no lower deviations, and
  it suffices to consider subsets of $[\sigma_0,\infty).$
  First, let $O\neq\emptyset$ be open,
  and $\lambda>0$ be arbitrary. We can pick $x>1$ and $\delta>0$ such that
  \[
    \inf O < \sigma_0\sqrt{x-\delta} < \sigma_0\sqrt{x+ \delta} < \inf O + \lambda
  \]
  and
  \[
    \big(\sigma_0\sqrt{x-\delta},\sigma_0\sqrt{x+\delta}\big) \subseteq O.
  \]
  Then,
  \begin{align*}
     P\left( \sup_{0<u<t}\frac{X_u}{h(u)} \in O \right)
     &\geq P\left( \sup_{0<u<t}\frac{X_u}{h(u)} \in \big(\sigma_0\sqrt{x-\delta},\sigma_0\sqrt{x+\delta}\big) \right) \\
     &= P\left( \sup_{0<u<t}\frac{X_u}{h(u)} \geq \sigma_0\sqrt{x-\delta})\right)
     -P\left( \sup_{0<u<t}\frac{X_u}{h(u)}\geq \sigma_0\sqrt{x+\delta}\right) \\
     &= e^{-(x- \delta-1) (\log\log\frac{1}{t})(1 + o(1))},
  \end{align*}
  by Theorem~\ref{thm:main eps}. Therefore,
  \begin{align*}
     \liminf_{t \searrow 0} \frac{1}{\log\log\frac{1}{t}} \log P\left( \sup_{0 < u < t} \frac{X_u}{h(u)} \in O \right)
     &\geq -(x-\delta-1)\\
    &\geq - \Big(\frac{\inf O +\lambda}{\sigma_0}\Big)^2 +1
    = -J(O)  +O(\lambda).
  \end{align*}
  Then taking $\lambda \searrow0$ yields~\eqref{eq:ldp lower}.
   Now let $C$ be closed. Recall that we may assume $C\subseteq[\sigma_0,\infty).$
   If $\inf C =\sigma_0,$ then $J(C)=0$ as $C$ is closed, and it suffices to estimate the probability
   in~\eqref{eq:ldp upper} by~$1$. Otherwise, let $\sigma_0\sqrt{1+\kappa} := \inf C$
   with $\kappa>0$. Then, by Theorem~\ref{thm:main eps},
   \begin{align*}
     \limsup_{t \searrow 0} \frac{1}{\log\log\frac{1}{t}} \log
       P\left( \sup_{0<u<t}\frac{X_u}{h(u)} \in C \right) 
       &\leq \limsup_{t \searrow 0} \frac{1}{\log\log\frac{1}{t}} \log
       P\left( \sup_{0<u<t}\frac{X_u}{h(u)} \geq\sigma_0\sqrt{1+\kappa} \right) \\
       &= -\kappa = -J(C).
   \end{align*}
\end{proof}

\bibliographystyle{siam}
\bibliography{literature}

\begin{thebibliography}{1}

\bibitem{BaCa11}
{\sc P.~Baldi and L.~Caramellino}, {\em General {F}reidlin-{W}entzell large
  deviations and positive diffusions}, Statist. Probab. Lett., 81 (2011),
  pp.~1218--1229.

\bibitem{Ca98}
{\sc L.~Caramellino}, {\em Strassen's law of the iterated logarithm for
  diffusion processes for small time}, Stochastic Process. Appl., 74 (1998),
  pp.~1--19.

\bibitem{DeZe98}
{\sc A.~Dembo and O.~Zeitouni}, {\em Large deviations techniques and
  applications}, vol.~38 of Stochastic Modelling and Applied Probability,
  Springer-Verlag, New York, second~ed., 1998.

\bibitem{KaSh91}
{\sc I.~Karatzas and S.~E. Shreve}, {\em Brownian motion and stochastic
  calculus}, vol.~113 of Graduate Texts in Mathematics, Springer-Verlag, New
  York, second~ed., 1991.

\bibitem{Le86}
{\sc H.~R. Lerche}, {\em Boundary crossing of {B}rownian motion}, vol.~40 of
  Lecture Notes in Statistics, Springer-Verlag, Berlin, 1986.

\bibitem{Mc69}
{\sc H.~P. McKean, Jr.}, {\em Stochastic integrals}, Probability and
  Mathematical Statistics, No. 5, Academic Press, New York-London, 1969.

\bibitem{Pi96}
{\sc V.~I. Piterbarg}, {\em Asymptotic methods in the theory of {G}aussian
  processes and fields}, vol.~148 of Translations of Mathematical Monographs,
  American Mathematical Society, Providence, RI, 1996.
\newblock Translated from the Russian by V. V. Piterbarg, Revised by the
  author.

\bibitem{St67}
{\sc V.~Strassen}, {\em Almost sure behavior of sums of independent random
  variables and martingales}, in Proc. {F}ifth {B}erkeley {S}ympos. {M}ath.
  {S}tatist. and {P}robability ({B}erkeley, {C}alif., 1965/66), Univ.
  California Press, Berkeley, Calif., 1967, pp.~315--343.

\end{thebibliography}

\end{document}